\documentclass[10.5pt,a4paper]{amsart}
\usepackage{etex} 
\usepackage[latin1]{inputenc}
\usepackage[T1]{fontenc}
\usepackage[right=3cm,left=3cm,top=2cm,bottom=2cm]{geometry}
\usepackage{ifthen}
\usepackage{enumerate}
\usepackage{amsmath}
\usepackage{amsthm}
\usepackage{amsfonts}
\usepackage{amssymb}
\usepackage{latexsym}
\usepackage{ae}
\usepackage{xcolor} 
\usepackage{graphics,amsmath,amssymb}
\usepackage[np]{numprint}
\usepackage{fourier} 
\usepackage{hyperref}

\npdecimalsign{\ensuremath{.}}

\title{Sums of certain fractional parts}
\author{Olivier Bordell\`{e}s}
\address{2 all\'{e}e de la combe \\ 43000 Aiguilhe \\ France}
\email{borde43@wanadoo.fr}
\date{}
\dedicatory{}

\theoremstyle{plain}
\newtheorem{theorem}{Theorem}

\newtheorem{lemma}[theorem]{Lemma}

\theoremstyle{definition}

\theoremstyle{remark}

\newcommand{\Z}{\mathbb {Z}}

\newcommand{\R}{\mathbb {R}}
\newcommand{\C}{\mathbb {C}}

\DeclareMathOperator{\card}{Card}

\makeatletter
\@namedef{subjclassname@2010}{%
  \textup{2010} Mathematics Subject Classification}
\makeatother

\begin{document}

\begin{abstract}
In this note, an upper bound for the sum of fractional parts of certain smooth functions is established. Such sums arise naturally in numerous problems of analytic number theory. The main feature is here an improvement of the main term due to the use of Weyl's bound for exponential sums and a device used by Popov.
\end{abstract}


\maketitle

\thispagestyle{myheadings}
\font\rms=cmr8 
\font\its=cmti8 
\font\bfs=cmbx8

\section{Introduction and main result}

\noindent
Le $\{x\}$ be the fractional part of $x \in \R$ and $\psi(x) := \{x\} - \frac{1}{2}$ be the first Bernoulli function. Sums of the shape
\begin{equation}
   \sum_{N < n \leqslant 2N} \psi(f(n)) \label{eq:sum}
\end{equation}
where $N$ is a large number and $f$ is a smooth function, are of great importance in analytic number theory (see \cite{borl,gra,mon} for instance). A large amount of problems are reduced to obtaining a non-trivial bound for the sum \eqref{eq:sum}, such as, among others, the Dirichlet divisor problem, the Gauss circle problem, the problem of the gaps between $k$-free numbers or the distribution of squarefull numbers. 

The general strategy is to use a truncated version of the expansion of $\psi$ in Fourier series, providing the inequality
\begin{equation}
   \sum_{N < n \leqslant 2N} \psi(f(n)) \ll \frac{N}{L} + \sum_{\ell \leqslant L} \frac{1}{\ell} \left | \sum_{N < n \leqslant 2N} e(\ell f(n)) \right| \label{eq:Erdos}
\end{equation}
where $e(x) := e^{2 i \pi x}$ and $L \geqslant 1$ is any integer parameter to be chosen optimally. The problem is henceforth reduced to estimating exponential sums, for which several methods have been developed by Weyl, van der Corput and Vinogradov. For instance, when $f \in C^k \left[ N,2N \right]$ satisfies $\left | f^{(k)} (x) \right | \asymp \lambda_k <1$, van der Corput's estimate (see \cite[Theorem~2.8]{gra}) and the inequality \eqref{eq:Erdos} yield the bound 
\begin{equation}
   \sum_{N < n \leqslant 2N} \psi(f(n)) \ll_k N \lambda_k^{\frac{1}{2^k-1}} + N^{1-2^{1-k}} \log N + N^{1-2^{3-k} + 2^{4-2k}} \lambda_k^{-2^{1-k}} \label{eq:vdc}
\end{equation}
when $k \geqslant 2$, the term $N^{1-2^{1-k}} \log N$ being removed in the case $k=2$. When $k \in \{2,3,4\}$, this respectively gives 
\begin{eqnarray*}
   & & \sum_{N < n \leqslant 2N} \psi(f(n)) \ll N \lambda_2^{1/3} + \lambda_2^{-1/2}, \\
   & & \\
   & & \sum_{N < n \leqslant 2N} \psi(f(n)) \ll N \lambda_3^{1/7} + N^{3/4} \log N + N^{1/4} \lambda_3^{-1/4}, \\
   & & \\
   & & \sum_{N < n \leqslant 2N} \psi(f(n)) \ll N \lambda_4^{1/15} + N^{7/8} \log N + N^{9/16} \lambda_4^{-1/8}.
\end{eqnarray*}
The term $N \lambda_k^{\frac{1}{2^k-1}}$ is usually called the \textit{main term}, the other two terms being the \textit{secondary terms}. For monomial functions, i.e. functions $f \in C^\infty \left[ N,2N \right]$ such that $\left | f^{(k)} (x) \right | \asymp TN^{-k}$ for some $T \geqslant 1$ and for any positive integer $k$, van der Corput's method of exponent pairs provides better results (see \cite[Lemma~4.3]{gra} or \cite[Corollary~6.35]{borl}). 

Recently, using new bounds given in \cite[Theorem~1.2]{woo} for the number $J_{s,k}(X)$ of integral solutions of the system
$$x_1^j + \dotsb + x_s^j = y_1^j + \dotsb + y_s^j \quad \left( 1 \leqslant j \leqslant k \right) $$
with $1 \leqslant x_i,y_i \leqslant X$ ($1 \leqslant i \leqslant s$), some improvements in exponential sums have appeared in the literature (see \cite{rob,hea}). If we use the main result in \cite{hea} combined with \eqref{eq:Erdos}, we obtain
\begin{equation}
   N^{-\varepsilon} \sum_{N <n \leqslant 2N} \psi \left( f(n) \right) \ll _{\varepsilon,k} N \lambda_k^{\frac{1}{k^2-k+1}} + N^{1 - \frac{1}{k(k-1)}} + N^{1 - \frac{2}{k(k-1)}} \lambda_k^{-\frac{2}{k^2(k-1)}} \label{eq:heath}
\end{equation}
where $k \in \Z_{\geqslant 3}$ and $f \in C^k \left[ N,2N\right]$ is such that there exists $\lambda_k \in \left( 0,1 \right)$ such that $\left | f^{(k)} \right | \asymp \lambda_k$. This improves on the main term of \eqref{eq:vdc} as soon as $k \geqslant 4$ and gives the same exponent when $k=3$:
\begin{eqnarray*}
   & & N^{-\varepsilon} \sum_{N < n \leqslant 2N} \psi(f(n)) \ll_{\varepsilon} N \lambda_3^{1/7} + N^{5/6} + N^{2/3} \lambda_3^{-1/9}, \\
   & & \\
   & & N^{-\varepsilon} \sum_{N < n \leqslant 2N} \psi(f(n)) \ll_{\varepsilon} N \lambda_4^{1/13} + N^{11/12} + N^{5/6} \lambda_4^{-1/24}.
\end{eqnarray*}

Any improvement of \eqref{eq:vdc} or \eqref{eq:heath} when $k \in \{2,3,4\}$ may lead to new results in the aforementioned problems. The main purpose of this note is to improve the main term in the cases $k \geqslant 2$ for \eqref{eq:vdc} and $k \in \{2,3,4,5 \}$ for \eqref{eq:heath}. To do this, we use Weyl's differencing method and add a device due to Popov \cite{pop}, also used in \cite{fom} to estimates the sums 
$$\sum_{N < n \leqslant 2N} \psi(P(n))$$
where $P(x)$ is a polynomial of degree $2$ or $3$ with small positive leading coefficient. The method was then generalized in \cite{bor} to any polynomial of degree $\geqslant 2$, and we use here the Weyl's schift to extend the results to smooth functions. Unfortunately, as often in exponential sums estimates, the secondary terms remain too weak to be really efficients in practice. Nevertheless, the result below seems to be new and we think that it may be of interest.

\begin{theorem}
\label{eq:th}
Let $k,N \in \Z_{\geqslant 2}$, $f \in C^{k+1} \left[ N,2N \right]$ such that there exist $\lambda_k, \lambda_{k+1}, s_k > 0$ and $c_k,c_{k+1}\geqslant 1$ such that, for any $x \in \left[ N,2N \right]$ and any $j \in \{k,k+1\}$
$$\lambda_j \leqslant \left| f^{(j)}(x) \right| \leqslant c_j \lambda_j \quad \text{with} \quad \lambda_k= s_k N\lambda_{k+1}.$$
Define $d_k:=2^k(k+1) + 2k$. Then, for any $\varepsilon > 0$
$$N^{-\varepsilon} \sum_{N < n \leqslant 2N} \psi \left( f(n) \right) \ll_{k,\varepsilon} N \lambda_k^{2^{1-k}} + N \lambda_k^{\frac{2^{1-k}}{k}}\left( N^k \lambda_k  \right)^{-\frac{2}{k d_k}} + N \left( N^k \lambda_k \right)^{-\frac{2}{d_k}}.$$
\end{theorem}

Note that, if $N^{k-1} \lambda_k^{2+2^{1-k}} \leqslant 1$, then the $2$nd term is absorbed by the $3$rd one. To compare with \eqref{eq:vdc} and \eqref{eq:heath} in the cases $k \in \{2,3,4\}$, Theorem~\ref{eq:th} respectively gives
\begin{eqnarray*}
   & & N^{-\varepsilon} \sum_{N < n \leqslant 2N} \psi \left( f(n) \right) \ll_{\varepsilon} N \lambda_2^{1/2} + N^{7/8} \lambda_2^{3/16} + N^{3/4} \lambda_2^{-1/8}, \\
   & & \\
   & & N^{-\varepsilon} \sum_{N < n \leqslant 2N} \psi \left( f(n) \right) \ll_{\varepsilon} N \lambda_3^{1/4} + N^{18/19} \lambda_3^{5/76} + N^{16/19} \lambda_3^{-1/19}, \\
   & & \\
   & & N^{-\varepsilon} \sum_{N < n \leqslant 2N} \psi \left( f(n) \right) \ll_{\varepsilon} N \lambda_4^{1/8} + N^{43/44} \lambda_4^{9/352} + N^{10/11} \lambda_4^{-1/44}.
\end{eqnarray*}

\section{Technical lemmas}

\begin{lemma}
\label{lem:sum}
Let $M \in \Z_{\geqslant 0}$, $N \in \Z_{\geqslant 1}$, $H \in \Z_{\geqslant 4}$ and $\alpha > 0$. Then
$$\sum_{M < n \leqslant M+N} \min \left( H , \frac{1}{\| n \alpha \|} \right) \ll HN \alpha + \left( N + \alpha^{-1} \right) \log H + H.$$
\end{lemma}

\begin{proof}
This is \cite[Lemma~3.2]{bor}.
\end{proof}

\begin{lemma}[Weyl's shift]
\label{lem:Weyl}
Let $N,N_1 \in \Z_{\geqslant 1}$ such that $N < N_1 \leqslant 2N$ and $a_{N+1},\dotsc,a_{N_1} \in \C$ satisfying $\left | a_n \right | \leqslant 1$. Then, for any $H \in \left\lbrace 1,\dotsc,N-1 \right\rbrace$
$$\left | \sum_{N < n \leqslant N_1} a_n \right | \leqslant \frac{1}{H} \left | \sum_{N < n \leqslant N_1-H} \ \sum_{h \leqslant H} a_{n+h} \right | + H.$$
\end{lemma}

\begin{proof}
Define
$$\alpha_n := \begin{cases} a_n, & \textrm{if\ } N < n \leqslant N_1 \\ 0, & \textrm{otherwise}. \end{cases}$$
Then
\begin{eqnarray*}
   \sum_{N < n \leqslant N_1} a_n &=& \frac{1}{H} \sum_{h \leqslant H} \sum_{n \in \Z} \alpha_n = \frac{1}{H} \sum_{h \leqslant H} \sum_{n \in \Z} \alpha_{n+h} = \frac{1}{H} \sum_{h \leqslant H} \ \sum_{N-h < n \leqslant N_1-h} \alpha_{n+h} \\
   &=& \frac{1}{H} \sum_{n \leqslant N_1-1} \ \sum_{h \leqslant H} \alpha_{n+h} - \frac{1}{H} \sum_{n \leqslant N-1} \ \sum_{h \leqslant H} \alpha_{n+h} \\
   &=& \frac{1}{H} \sum_{n \leqslant N_1-H} \ \sum_{h \leqslant H} \alpha_{n+h} + \frac{1}{H} \sum_{N_1-H < n \leqslant N_1-1} \ \sum_{h \leqslant H} \alpha_{n+h} - \frac{1}{H} \sum_{n \leqslant N-1} \ \sum_{h \leqslant H} \alpha_{n+h} \\
   &=& \frac{1}{H} \sum_{N < n \leqslant N_1-H} \ \sum_{h \leqslant H} \alpha_{n+h} + \frac{1}{H} \sum_{N_1-H < n \leqslant N_1-1} \ \sum_{h \leqslant H} \alpha_{n+h} \\
\end{eqnarray*}
Since, in the first sum, $\alpha_{n+h} = a_{n+h}$, we get
$$\left | \sum_{N < n \leqslant N_1} a_n \right | \leqslant \frac{1}{H} \left | \sum_{N < n \leqslant N_1-H} \ \sum_{h \leqslant H} a_{n+h} \right | + \frac{H^2}{H}$$
as asserted. 
\end{proof}

\section{Proof of Theorem~\ref{eq:th}}

One may assume $N^{-k}< \lambda_k < 1$, otherwise
$$N \lambda_k^{2^{1-k}} + N \left( N^k \lambda_k \right)^{-2/d_k} > N.$$
Using \eqref{eq:Erdos} and Lemma~\ref{lem:Weyl}, we get for any $H,L \in \Z_{\geqslant 1}$ such that $H < N$
\begin{eqnarray*}
   \sum_{N < n \leqslant 2N} \psi \left( f(n) \right) & \ll & \frac{N}{L} + \sum_{\ell \leqslant L} \frac{1}{\ell} \left | \sum_{N < n \leqslant 2N} e \left( \ell f(n) \right) \right | \\
   & \ll & \frac{N}{L} + \sum_{\ell \leqslant L} \frac{1}{\ell} \left\lbrace \frac{1}{H} \sum_{N < n \leqslant 2N-H} \left | \sum_{h \leqslant H} e \left( \ell f(n+h)\right) \right | + H \right\rbrace \\
   & \ll & \frac{N}{L} + \frac{1}{H} \sum_{N < n \leqslant 2N-H} \; \sum_{\ell \leqslant L} \frac{1}{\ell} \left | \sum_{h \leqslant H} e \left( \ell f(n+h)\right) \right | + H \log L \\
   & \ll & \frac{N}{L}+ \frac{1}{H} \sum_{N < n \leqslant 2N} \sum_{\ell \leqslant L} \frac{1}{\ell} \left | \sum_{h \leqslant H} e \left( \ell g_n(h) + \ell r_{n,k}(h) \right) \right | + H \log L
\end{eqnarray*}
with $g_n(h) := \frac{1}{k!}h^k f^{(k)} (n) + \dotsb + h f^{\,\prime }(n)$ and
$$r_{n,k}(h) := \frac{1}{k!} \int_n^{n+h} \left( n+h-t\right)^k f^{(k+1)} (t) \, \textrm{d}t.$$
Note that
$$\sum_{h \leqslant H-1} \left | r_{n,k}(h+1) - r_{n,k}(h) \right | < \frac{c_{k+1}}{(k+1)!} H^{k+1} \lambda_{k+1} = \frac{c_{k+1}}{s_k(k+1)!} H^{k+1} N^{-1}\lambda_{k}$$
so that, by partial summation
$$\sum_{N < n \leqslant 2N} \psi \left( f(n) \right) \ll \frac{N}{L} + \frac{1}{H} \sum_{N < n \leqslant 2N} \left\{ \sum_{\ell \leqslant L} \left( \frac{1}{\ell} + H^{k+1}N^{-1}\lambda_{k} \right) \underset{H_1\leqslant H}{\max} \left| \sum_{h \leqslant H_1} e\left( \ell g_n(h) \right) \right| \right\} + H \log L.$$
Now assume
\begin{equation}
   L \leqslant H^{-k-1} N \lambda_{k}^{-1} \label{eq:hyp}
\end{equation}
so that
\begin{equation}
   \sum_{N < n \leqslant 2N} \psi \left( f(n) \right) \ll \frac{N}{L} + \frac{1}{H}\sum_{N < n \leqslant 2N} \mathcal{S}_{H,L} (n) + H \log L \label{eq:firstbound}
\end{equation}
where
$$\mathcal{S}_{H,L} (n):= \sum_{\ell \leqslant L} \frac{1}{\ell} \, \underset{H_1\leqslant H}{\max} \left| \sum_{h \leqslant H_1} e\left( \ell g_n(h) \right) \right|$$
and set $\alpha_{n,k} := \frac{1}{k!} \left | f^{(k)} (n) \right |$. From Weyl's bound \cite[p. 42]{mon}, we get for any $\varepsilon > 0$
\begin{eqnarray*}
   & & \mathcal{S}_{H,L} (n) \\
   & \ll & \sum_{\ell \leqslant L} \frac{1}{\ell} \, \underset{H_1\leqslant H}{\max} \left( H_1^{2^{k-1}-1} + H_1^{2^{k-1}-k+\varepsilon} \sum_{h \leqslant k!H_1^{k-1}} \min \left( H_1 , \frac{1}{\| h \ell \alpha_{n,k} \|} \right) \right)^{2^{1-k}} \\
   & \ll & \sum_{\ell \leqslant L} \frac{1}{\ell} \left( H^{2^{k-1}-1} + H^{2^{k-1}-k+\varepsilon} \sum_{h \leqslant k!H^{k-1}} \min \left( H , \frac{1}{\| h \ell \alpha_{n,k} \|} \right) \right)^{2^{1-k}} \\
 \\
  & \ll & H^{1-2^{1-k}} \log L + H^{1 - k2^{1-k}+\varepsilon} \sum_{\ell \leqslant L} \frac{1}{\ell} \left( \sum_{h \leqslant k!H^{k-1}} \min \left( H , \frac{1}{\| h \ell \alpha_{n,k} \|} \right) \right)^{2^{1-k}} \\
\end{eqnarray*}
and H\"{o}lder's inequality applied with exponent $\frac{2^k}{2^k-2}$ yields
\begin{eqnarray*}
  & & \mathcal{S}_{H,L} (n) \\
  & \ll & H^{1-2^{1-k}} \log L + H^{1 - k2^{1-k}+\varepsilon} \left( \sum_{\ell \leqslant L} \frac{1}{\ell} \right)^{1-2^{1-k}} \left( \sum_{\ell \leqslant L} \frac{1}{\ell} \sum_{h \leqslant k!H^{k-1}} \min \left( H , \frac{1}{\| h \ell \alpha_{n,k} \|} \right) \right)^{2^{1-k}} \\
  & \ll & H^{1-2^{1-k}} \log L + H^{1 - k2^{1-k}+\varepsilon} (\log L)^{1-2^{1-k}} \left( \sum_{m \leqslant k!LH^{k-1}} \min \left( H , \frac{1}{\| m \alpha_{n,k} \|} \right) \sum_{\substack{\ell \mid m \\ \ell \leqslant L \\ m/\ell \leqslant k!H^{k-1}}} \frac{1}{\ell} \right)^{2^{1-k}} \\
  & \ll & H^{1-2^{1-k}} \log L \\
  & & {} + H^{1 - k2^{1-k}+\varepsilon} (\log L)^{1-2^{1-k}} \left( \sum_{j=0}^{k!L-1} \sum_{jH^{k-1} < m \leqslant (j+1)H^{k-1}} \min \left( H , \frac{1}{\| m \alpha_{n,k} \|} \right) \sum_{\substack{\ell \mid m \\ \ell \leqslant L \\ m/\ell \leqslant k!H^{k-1}}} \frac{1}{\ell} \right)^{2^{1-k}}.
\end{eqnarray*}
Following \cite[(13),(14)]{pop} (see also \cite{fom}), notice that, in the innersum, we have 
$$\frac{1}{\ell} \leqslant \frac{k!H^{k-1}}{m} < \frac{k!}{j} \quad \left( j \geqslant 1 \right)$$
so that
\begin{eqnarray*}
   (HL)^{-\varepsilon} \mathcal{S}_{H,L} (n) & \ll & H^{1-2^{1-k}}  + H^{1 - k2^{1-k}} \left( \sum_{m \leqslant H^{k-1}} \min \left( H , \frac{1}{\| m \alpha_{n,k} \|} \right) \frac{\sigma(m)}{m} \right)^{2^{1-k}} \\
   & &  + {} H^{1 - k2^{1-k}}\left( \sum_{j=1}^{k!L-1} \frac{1}{j} \sum_{jH^{k-1} < m \leqslant (j+1)H^{k-1}} \tau(m) \min \left( H , \frac{1}{\| m \alpha_{n,k} \|} \right) \right)^{2^{1-k}}
\end{eqnarray*}
and Lemma~\ref{lem:sum} and the crude bounds $\sigma(m) \ll m \log \log m$ and $\tau(m) \ll m^{\varepsilon}$ imply that
\begin{eqnarray*}
   \left( HL \right)^{-\varepsilon} \mathcal{S}_{H,L} (n)  & \ll & H^{1-2^{1-k}}  + H^{1 - k2^{1-k}} \left( H^k \alpha_{n,k} + H^{k-1} + \alpha_{n,k}^{-1} \right)^{2^{1-k}} \\
   & \ll & H \alpha_{n,k}^{2^{1-k}} + H^{1-2^{1-k}} + H^{1 - k2^{1-k}} \alpha_{n,k}^{-2^{1-k}}.
\end{eqnarray*}
Inserting in \eqref{eq:firstbound} and using $\alpha_{n,k} \asymp \lambda_k$, we get
\begin{eqnarray*}
   \sum_{N < n \leqslant 2N} \psi \left( f(n) \right) & \ll & \frac{N}{L} + \frac{\left( HL \right)^{\varepsilon}}{H}\sum_{N < n \leqslant 2N} \left( H \alpha_{n,k}^{2^{1-k}} + H^{1-2^{1-k}} + H^{1 - k2^{1-k}} \alpha_{n,k}^{-2^{1-k}} \right) + H \log L \\
   & \ll & \frac{N}{L} + \left( HL \right)^{\varepsilon} \left( N \lambda_k^{2^{1-k}} + NH^{-2^{1-k}} + NH^{-k2^{1-k}} \lambda_k^{-2^{1-k}} \right) + H \log L.
\end{eqnarray*}
Considering \eqref{eq:hyp}, the choice of $L =  \left \lfloor H^{-k-1} N \lambda_{k}^{-1} \right \rfloor$ gives
$$\sum_{N < n \leqslant 2N} \psi \left( f(n) \right) \ll H^{k+1} \lambda_k + \left( HN \lambda_{k}^{-1} \right)^{\varepsilon} \left( N \lambda_k^{2^{1-k}} + NH^{-2^{1-k}} + NH^{-k2^{1-k}} \lambda_k^{-2^{1-k}} \right) + H \log N.$$
The asserted result follows by choosing $H = 4\left \lfloor \left( N^{2^k} \lambda_k^{-(2^k+2)} \right)^{1/d_k} \right \rfloor$, the extra term $H \log N$ being absorbed by the term $N^{1+\varepsilon} \left( N^k \lambda_k \right)^{-2/d_k}$ since $N^k \lambda_k >1$. Note that this latter hypothesis also ensures that $1 < L < N^{k+1}$ and $H < N$, completing the proof.
\qed

\section{Extension to integer points close to smooth curves}

\noindent
In this section, let $\delta \in \left( 0,\frac{1}{4} \right] $, $N \in \Z_{\geqslant 2}$ large, $f: \left[  N,2N \right]  \longrightarrow \R$ be any function, and define
\begin{equation}
   \mathcal{R}(f,N,\delta):= \card \left\lbrace n \in \left[  N,2N \right] \cap \Z: \| f(n) \| < \delta \right\rbrace. \label{eq:R}
\end{equation}
Since it is known (see \cite[Exercise~4 p. 350]{borl} for instance) that, for any integer $1 \leqslant L \leqslant 1 + \left \lfloor (8 \delta)^{-1} \right \rfloor$, we have
$$\mathcal{R}(f,N,\delta) \ll \frac{N}{L} + \frac{1}{L} \sum_{\ell=1}^{L-1} \left | \sum_{N \leqslant n \leqslant 2N} e\left (\ell f(n) \right) \right|$$
the proof of Theorem~\ref{eq:th} may easily be adapted in a similar way to get a bound for $\mathcal{R}(f,N,\delta)$ of the same kind.

\begin{theorem}
\label{eq:th2}
Let $\delta \in \left( 0,\frac{1}{4} \right] $, $k,N \in \Z_{\geqslant 2}$, $f \in C^{k+1} \left[ N,2N \right]$ such that there exist $\lambda_k, \lambda_{k+1}, s_k > 0$ and $c_k,c_{k+1}\geqslant 1$ such that, for any $x \in \left[ N,2N \right]$ and any $j \in \{k,k+1\}$
$$\lambda_j \leqslant \left| f^{(j)}(x) \right| \leqslant c_j \lambda_j \quad \text{with} \quad \lambda_k= s_k N\lambda_{k+1}.$$
Define $d_k:=2^k(k+1) + 2k$. Then, for any $\varepsilon > 0$
$$\mathcal{R}(f,N,\delta) \ll_{k,\varepsilon} N \delta + N^{\varepsilon} \left( N \lambda_k^{2^{1-k}} + N \lambda_k^{\frac{2^{1-k}}{k}}\left( N^k \lambda_k  \right)^{-\frac{2}{k d_k}} + N \left( N^k \lambda_k \right)^{-\frac{2}{d_k}} \right).$$
\end{theorem}

\noindent
This must be compared to the existing results of the theory. For instance, under the hypothesis $\left |f^{\, \prime \prime} \right| \asymp \lambda_2$, it is proved in \cite{bra} that
$$\mathcal{R}(f,N,\delta) \ll N \lambda_2^{1/3} +  N \delta + \left( \frac{\delta}{\lambda_2} \right)^{1/2} +  1.$$
In the cases $k=3$ or $k=4$, it is known from \cite{huxs2} that
$$\mathcal{R} \left (f,N,\delta\right) \ll N \lambda_3^{1/6} + N \delta^{2/3} + N \left( \delta^3 \lambda_3 \right)^{1/12} + \left( \frac{\delta}{\lambda_2} \right)^{1/2} + 1$$
if $\left |f^{(j)} \right| \asymp \lambda_j$ for $j \in \{2,3\}$ such that $\lambda_2 = N \lambda_3$, and  
$$\mathcal{R} \left (f,N,\delta\right) \ll N \lambda_4^{1/10} + N \delta^{1/3} + N \left( \delta^3 \lambda_4 \right)^{1/21} + \left( \frac{\delta}{\lambda_3} \right)^{1/3} + 1$$
if $\left |f^{(j)} \right| \asymp \lambda_j$ for $j \in \{3,4\}$ such that $\lambda_3 = N \lambda_4$.

\end{document}